\documentclass[12pt]{amsart}
\usepackage{amssymb}

\newtheorem{theorem}{Theorem}[section]
\newtheorem{lemma}[theorem]{Lemma}
\newtheorem{corollary}[theorem]{Corollary}
\newtheorem{proposition}[theorem]{Proposition}

\theoremstyle{definition}
\newtheorem{remark}[theorem]{Remark}

\numberwithin{equation}{section}

\begin{document}

\baselineskip=15.5pt

\title[Fundamental group of a GIT quotient]{Fundamental group of a
geometric invariant theoretic quotient}

\author[I. Biswas]{Indranil Biswas}

\address{School of Mathematics, Tata Institute of Fundamental
Research, Homi Bhabha Road, Bombay 400005, India}

\email{indranil@math.tifr.res.in}

\author[A. Hogadi]{Amit Hogadi}

\address{School of Mathematics, Tata Institute of Fundamental
Research, Homi Bhabha Road, Bombay 400005, India}

\email{amit@math.tifr.res.in}

\author[A.J. Parameswaran]{A. J. Parameswaran}

\address{School of Mathematics, Tata Institute of Fundamental
Research, Homi Bhabha Road, Bombay 400005, India}

\email{param@math.tifr.res.in}



\date{}

\begin{abstract}
Let $M$ be an irreducible smooth projective variety, defined over an algebraically
closed field, equipped with an action of a connected reductive affine algebraic group 
$G$, and let ${\mathcal L}$ be a $G$--equivariant very ample line bundle on
$M$. Assume that the GIT quotient $M/\!\!/G$ is a nonempty set.
We prove that the homomorphism of algebraic fundamental groups
$\pi_1(M)\, \longrightarrow\, \pi_1(M/\!\!/G)$, induced by the rational map $M\,
\dasharrow\, M/\!\!/G$, is an isomorphism.

If $k\,=\, \mathbb C$, then we show that the above rational map $M\, \dasharrow
\, M/\!\!/G$ induces an isomorphism between the topological fundamental groups.
\end{abstract}

\maketitle

\section{Introduction}

Let $M$ be an irreducible smooth projective variety defined over an algebraically
closed field $k$ of arbitrary characteristic. Fix a very ample line bundle
$\mathcal L$ on $M$. Let $G$ be a connected reductive affine algebraic group over $k$
acting algebraically on both $M$ and $\mathcal L$ such that $\mathcal L$ is a
$G$--equivariant line bundle. The geometric invariant theoretic quotient
$M/\!\!/G$ for this action of $G$ on $(M\, ,{\mathcal L})$ will be denoted by
$X$. We assume that $X$ is nonempty. The quotienting produces a rational
morphism $M\, \dasharrow\, X$. This rational morphism produces a homomorphism
$$
h\, :\, \pi_1(M)\, \longrightarrow\, \pi_1(X)
$$
between the algebraic fundamental groups.

When $k\, =\, \mathbb C$, the above rational morphism $M\, \dasharrow\, X$
produces a homomorphism
$$
h'\, :\, \pi^t_1(M)\, \longrightarrow\, \pi^t_1(X)
$$
between the topological fundamental groups.

We prove the following (see Theorem \ref{thm1} and Theorem \ref{thm2}):

\begin{theorem}\label{thm0}
Let $G/k$ be a connected reductive algebraic affine group acting on a smooth connected
projective variety $M/k$.
\begin{enumerate}
\item The above homomorphism $h$ between the algebraic fundamental groups is an 
isomorphism.

\item If $k\,=\,\mathbb C$, the homomorphism $h'$ between the topological 
fundamental groups is an isomorphism.
\end{enumerate}
\end{theorem}

Hui Li proved that for an Hamiltonian action of a compact group on a compact
symplectic manifold $M$, the fundamental group of the symplectic quotient coincides
with $\pi_1(M)$ \cite[p. 346, Theorems 1.2, 1.3]{Li}. A theorem of Kirwan and
Kempf--Ness says that for a linear action of a complex reductive group on a smooth
complex projective variety, if the stabilizer of every semistable point is finite, then
the GIT quotient is homeomorphic to the symplectic quotient \cite[Theorem 7.5,
Remark 8.14]{Ki}, \cite{KN}.

\section{Lifting an action}

Let $k$ be an algebraically closed field. Let $G$ be a connected reductive affine
algebraic group over $k$ and $M/k$ an irreducible smooth projective variety equipped
with an algebraic action
$$
\theta\, :\, M\times G\, \longrightarrow\, M
$$
of $G$.
Continuing with the above notation we have the following proposition. 

\begin{proposition}\label{prop1}
Let $\varphi\, :\, M'\, \longrightarrow\, M$ be an \'etale Galois covering
morphism such that $M'$ is connected. Then there is a unique action of $G$ on $M'$
$$
\theta'\, :\, M'\times G\, \longrightarrow\, M'
$$
that lifts $\theta$, meaning
$\theta\circ (\varphi\times {\rm Id}_G)\,=\,\varphi\circ \theta'$.
\end{proposition}

\begin{proof}
Since $M'\times G$ is connected, for any two lifts $\theta_1\, , \theta_2\, :\, M'
\times G\, \longrightarrow\, M'$ of the morphism $\theta$, there is a deck
transformation $\gamma\, \in\, \text{Gal}(\varphi)$ such that $\theta_2\,=\,
\gamma\circ \theta_1$. On the other hand, for an action $\theta'$ of $G$ on $M'$, we
have $\theta'(z\, ,e)\, =\, z$ for all $z\, \in\, M'$, where $e\, \in\, G$ denotes
the identity element. Therefore, there can be at most one action $\theta'$ of $G$
satisfying the condition that it lifts $\theta$.

We will now prove the existence of a lifted action.
For any closed point $x\, \in\, M$, let
$$
\varphi^x\, :\, G\, \longrightarrow\, M
$$
be the morphism defined by $g\, \longmapsto\, \theta (x\, , g)$. Consider
the induced homomorphism of algebraic fundamental groups
\begin{equation}\label{e1}
\varphi^x_*\, :\, \pi_1(G,\, e)\, \longrightarrow\, \pi_1(M, \,x)\, .
\end{equation}
We will show that
\begin{equation}\label{e2}
\varphi^x_*\,=\, 0\, .
\end{equation}

To prove \eqref{e2}, first take $x$ to be such that its orbit $\theta(x\, , 
,G)\,\subset\, M$ is of minimal dimension among all the $G$--orbits in $M$. 
Since the boundary of $\theta(x\, ,G)$
$$
\overline{\theta(x\, ,G)} \setminus \theta(x\, ,G)\,\subset\, M
$$
is preserved by the action of $G$, the condition that the
dimension of $\theta(x\, ,G)$ is the minimum one implies that this boundary
is empty. In other words, the orbit $\theta(x\, ,G)$ is a complete subvariety 
of $M$. Let $G_x\, \subset\, G$ be the isotropy group-scheme for the point $x$
(we note that $G_x$ need not be reduced). The reduced group
\begin{equation}\label{ing}
P\, :=\, G_{x, {\rm red}}\, \subset\, G_x\, \subset\, G
\end{equation}
is a parabolic subgroup of $G$ because $\theta(x\, ,G)$ is complete. Consider the
natural projection 
\begin{equation}\label{xi}
\xi\, :\, P\backslash G\, \longrightarrow\, G_x\backslash G\,=\, \theta(x\, ,G)
\end{equation}
given by the first inclusion in \eqref{ing}. Any solvable algebraic group defined over $k$ is
isomorphic as a variety (not as an algebraic group) to a product of copies of ${\mathbb A}^1_k$ and
${\mathbb G}_m$ (recall that $k$ is algebraically closed). Therefore, from the Bruhat decomposition
of $G$ it
follows that the variety $G$ is rational. Combining this with the fact that the quotient morphism $G\,
\longrightarrow\, P\backslash G$ is separable we conclude that $P\backslash G$ is separably rationally
connected. This implies that the variety $P\backslash G$ is simply connected \cite[p. 75, Theorem 13]{Ko}.

Let
$$
\nu\, :\, \widetilde{C}\,\longrightarrow\, \theta(x\, ,G)
$$
be a finite \'etale Galois covering. Let
\begin{equation}\label{c1}
\begin{matrix}
(P\backslash G)\times_{\theta(x\, ,G)} \widetilde{C} & \stackrel{\delta}{\longrightarrow}
& \widetilde{C}\\
~ \Big\downarrow \mu && \Big\downarrow\\
P\backslash G& \stackrel{\xi}{\longrightarrow} &  \theta(x\, ,G)
\end{matrix}
\end{equation}
be the pullback of the covering $\nu$ to $P\backslash G$ by the morphism $\xi$ in \eqref{xi}. Since
$P\backslash G$ is simply connected, the covering $\mu$ in \eqref{c1} admits a section
$$
\eta\, :\, P\backslash G\, \longrightarrow\, (P\backslash G)\times_{\theta(x\, ,G)} \widetilde{C}\, .
$$
Now consider the composition
$$
\delta\circ \eta\, :\, P\backslash G\, \longrightarrow\, \widetilde{C}\, ,
$$
where $\delta$ is the projection in \eqref{c1}. The image of $\delta\circ \eta$ is a connected
component of $\widetilde{C}$ that projects isomorphically to $\theta(x\, ,G)$. Indeed,
this follows immediately from the fact that $\xi$ in \eqref{xi} is bijective
on closed points. Hence we conclude that
$$
\pi_1(\theta(x\, ,G),\, x)\,=\, 0\, .
$$
Therefore, \eqref{e2} holds for this point $x$.

For another closed point $y$ of $M$, let $\varphi^y\, :\, G\, \longrightarrow\, 
M$ be the morphism defined by $g\, \longmapsto\, \theta (y\, , g)$. Let
$$
\psi^x\, :\, G\,\longrightarrow\, M\times G \   ~ \text{ and }\  ~
\psi^y\, :\, G\,\longrightarrow\, M\times G
$$
be the embeddings defined by $g\,\longmapsto\, (x\, ,g)$ and
$g\,\longmapsto\, (y\, ,g)$ respectively. So, we 
have $\varphi^x\,=\,\theta\circ\psi^x$
and $\varphi^y\,=\,\theta\circ\psi^y$. Therefore, 
$\varphi^x_*\,=\,\theta_*\circ\psi^x_*$
and $\varphi^y_*\,=\,\theta_*\circ\psi^y_*$, where $\varphi^x_*$ is the
homomorphism in \eqref{e1}. Since $M$ is projective,
$$\pi_1(M\times G,\, (y\, ,e))\,=\, \pi_1(M,\, y)\times \pi_1(G,\, e)$$
\cite[Expos\'e~X, \S~1, Corollaire~5.1]{sga1}. The two groups
$\pi_1(M\times G,\, (y\, ,e))$ and $\pi_1(M\times G,\, (x\, ,e))$ are identified
up to conjugation, and such an identification takes
the image $\psi^x_*(\pi_1(G,\, e))$
to $\psi^y_*(\pi_1(G,\, e))$. These imply that \eqref{e2} holds for the point $y$
because it holds for the point $x$. Therefore, \eqref{e2} holds for all $x$.

Take $M'$ in the statement of the proposition. Since $M'$ is projective, we have
$$\pi_1(M'\times G,\, (x'\, ,e))\,=\, \pi_1(M',\, x')\times \pi_1(G,\, e)$$
\cite[Expos\'e~X, \S~1, Corollaire~5.1]{sga1}. Consider the diagram
$$
\begin{matrix}
M'\times G && M'\\
~\,~\, ~\, ~\, ~\, ~\, ~\, ~\, ~\, ~\,~\,
\Big\downarrow \varphi\times {\rm Id}_G && ~\Big\downarrow \varphi\\
M\times G &\stackrel{\theta}{\longrightarrow} & M
\end{matrix}
$$
Take any closed point $x'\, \in\, \varphi^{-1}(x)$. Now consider the induced 
homomorphism of algebraic fundamental groups
$$
(\theta\circ (\varphi\times {\rm Id}_G))_*\, :\, \pi_1(M'\times G,\, (x'\, ,e))
\,=\, \pi_1(M',\, x')\times \pi_1(G,\, e)\,\longrightarrow\, \pi_1(M,\, x)\, .
$$
{}From \eqref{e2} it follows immediately that we have
$$
(\theta\circ (\varphi\times {\rm Id}_G))_*(\{e_0\}\times \pi_1(G,\, e))\,=\, 0\, ,
$$
where $e_0\, \in\, \pi_1(M',\, x')$ denotes the identity element. Consequently,
the image of $(\theta\circ (\varphi\times {\rm Id}_G))_*$
coincides with the image of the homomorphism
$$
\varphi_*\, :\, \pi_1(M', \,x')\,\longrightarrow\, \pi_1(M,\, x)\, .
$$
This implies that the pulled back Galois \'etale covering $$(M'\times G)\times_M 
M' \,\longrightarrow\, M'\times G$$ is identified with the trivial covering
$M'\times G\times \varphi^{-1}(x)\,\longrightarrow\, M'\times G$.
Consequently, there is a unique morphism
\begin{equation}\label{e3}
\theta'\, :\, M'\times G\, \longrightarrow\, M'
\end{equation}
such that the following two conditions hold:
\begin{equation}\label{g3}
\varphi\circ \theta'\,=\, \theta\circ (\varphi\times {\rm Id}_G)
\end{equation}
and $\theta'(x'\, , e)\,=\, x'$.

In view of \eqref{g3}, the morphism $\theta'_e\, :\, M'\, \longrightarrow\, M'$
defined by $y\, \longmapsto\,\theta'(y\, , e)$ is a lift of the identity map of $M$
because $\theta(z\, ,e)\,=\, z$ for all $z\, \in\, M$. Therefore, from the given
condition that $\theta'(x'\, , e)\,=\, x'$ it follows that $\theta'_e\,=
\,\text{Id}_{M'}$.

Next consider the two morphisms
$$
a\, , b\, :\, M'\times G\times G \, \longrightarrow\, M'
$$
defined by $(z\, , g_1\, ,g_2)\, \longmapsto\, \theta'(\theta'(z\, ,g_1)\, ,g_2)$
and $(z\, , g_1\, ,g_2)\, \longmapsto\, \theta'(z\, ,g_1g_2)$ respectively. The morphism
$a$ (respectively, $b$) is a lift of the morphism
$M\times G\times G \, \longrightarrow\, M$ defined by $(z\, , g_1\, ,g_2)\, \longmapsto\,
\theta(\theta(z\, ,g_1)\, ,g_2)$ (respectively, $(z\, , g_1\, ,g_2)\, \longmapsto\,
\theta(z\, ,g_1g_2)$). These two morphisms from $M\times G\times G$ to $M$ coincide
because $\theta$ is an action of $G$ on $M$. Also,
$$
a(z\, ,e\, ,e)\,=\, z\,=\, b(z\, ,e\, ,e)
$$
for all $z\, \in\, M'$. Therefore, we conclude that $a\,=\, b$. Consequently,
the morphism $\theta'$ defines an action of $G$ on $M'$.
\end{proof}

Let $\Gamma\, =\, \text{Gal}(\varphi)\, \subset\, \text{Aut}(M')$ be the Galois group
for the Galois covering $\varphi$.

\begin{lemma}\label{lem1}
The Galois action of $\Gamma$ on $M'$ commutes with the action of $G$ on $M'$ given
by $\theta'$ in Proposition \ref{prop1}.
\end{lemma}

\begin{proof}
Take any $\gamma\, \in\, \Gamma$. The morphism
$$
\gamma''\, :\, M'\times G\, \longrightarrow\, M'\, , ~\ (z\, ,g)\, \longmapsto\,
\gamma(\theta'(\gamma^{-1}(z)\, ,g))
$$
is an action of $G$ on $M'$ that lifts the action $\theta$ of $G$ on $M$. Now from
the uniqueness of $\theta'$ it follows that $\gamma''\,=\, \theta'$. This immediately
implies that the actions of $\Gamma$ and $G$ on $M'$ commute.
\end{proof}

\section{Fundamental group of the quotient}

Let $\mathcal L$ be a $G$--equivariant very ample line bundle on $M$. The action of any
$g\, \in\, G$ on any $v\, \in\, \mathcal L$ will be denoted by $v\cdot g$. Let
\begin{equation}\label{e4}
X\, :=\, M/\!\!/G
\end{equation}
be the geometric invariant theoretic (GIT) quotient of $M$ for the action of $G$ on
$(M\, ,{\mathcal L})$ \cite{Mu}. We assume that $X$ is nonempty.
This $X$ is an irreducible normal projective variety.
Let $U\, \subset\, M$ be the largest Zariski open subset over which the rational
map to the GIT quotient $$M\,\dasharrow\, X$$ is defined. Consider the homomorphism
\begin{equation}\label{e8}
\pi_1(U,\, u_0)\, \longrightarrow\, \pi_1(X,\, x_0)
\end{equation}
induced by the quotient map, where $u_0\, \in\, U$ is a point 
lying over a point $x_0\, \in\, X$. The codimension of the complement $M\setminus U
\, \subset\, M$ is at least two because this complement is the
indeterminacy locus of a rational morphism. Since $M$ is smooth, this
codimension condition implies that the homomorphism
$$
\pi_1(U,\, u_0)\, \longrightarrow\, \pi_1(M,\, u_0)\, ,
$$
induced by the inclusion map $U\, \hookrightarrow\, M$, is an isomorphism. Using
this isomorphism, the homomorphism in \eqref{e8} produces a homomorphism
\begin{equation}\label{f1}
h\, :\, \pi_1(M,\, u_0)\, \longrightarrow\, \pi_1(X,\, x_0)\, .
\end{equation}

Take $(M'\, , \varphi)$ as in Proposition \ref{prop1}.
Consider the ample line bundle $\varphi^*\mathcal L$ on the covering $M'$ and
the action $\theta'$ of $G$ on $M'$ (constructed in Proposition \ref{prop1}). Since
$\theta'$ is a lift of the action $\theta$, the action of $G$ on $\mathcal L$
produces an action of $G$ on $\varphi^*\mathcal L$. The action
of any $g\, \in\, G$ sends any $v\, \in\, (\varphi^*\mathcal L)_x$ to the element
in $(\varphi^*\mathcal L)_{\theta'(x,g)}$ that corresponds to $v\cdot g$
by the natural identification ${\mathcal L}_{\theta(\varphi(x),g)}\,=\,
(\varphi^*\mathcal L)_{\theta'(x,g)}$ after we consider $v$ as an element of
${\mathcal L}_{\varphi(x)}$ using the identification $(\varphi^*\mathcal L)_{x}\,=\,
{\mathcal L}_{\varphi(x)}$. This action of $G$
on $\varphi^*\mathcal L$ evidently lifts the action $\theta'$. Let 
\begin{equation}\label{e5}
{\widetilde X}'\, :=\, M'/\!\!/G
\end{equation}
be the GIT quotient for the action of $G$ on $(M'\, ,\varphi^*{\mathcal L})$.

Let $M^{ss}\,\subset\, M$ be the semistable locus for the action of $G$; it
is an open subscheme of $M$. Let
$$
{\widetilde M}\, :=\, \varphi^{-1}(M^{ss})\, \subset\, M'
$$
be the inverse image. We note that the subset ${\widetilde M}$ is left invariant
under the action of $G$ on $M'$ because $M^{ss}$ is preserved by the action of $G$
on $M$ and $\varphi$ is $G$--equivariant.

There is a finite collection of $G$--invariant nonzero sections
$\{s_i\}_{i=1}^N$ of ${\mathcal L}$ such that the collection
$$
U_i\,:=\,{\rm Spec}\, A_i \, =\, \{z\, \in\, M\,\mid\, s_i(z)\, \not=\, 0\}
$$
is an affine open cover of $M^{ss}$. Since $s_i$ is $G$--invariant, the subset
$U_i$ is preserved by the action of $G$ on $M$. The GIT quotient
$X\,=\,M/\!\!/G$ is obtained by patching together the
affine open subschemes $V_i\,=\, {\rm Spec}(A_i^G)$ (see \cite[Ch.~3, \S~3]{New}
and \cite[Ch.~3, \S~4]{New} for affine and projective GIT quotients respectively).

Consider the affine open cover of ${\widetilde M}$
$$
U'_i\,:= \,\varphi^{-1}(U_i)\,=\, {\rm Spec}\, B_i\,=\,
\{z\, \in\, {\widetilde M}\,\mid\, \varphi^*s_i(z)\, \not=\, 0\}\,
\subset\, {\widetilde M}\, .
$$
Note that each $U'_i$ is preserved by the action of $G$ on $M'$.
We may patch together the affine open subschemes $V'_i\,:=\,
{\rm Spec} (B_i^G)$ to construct a quotient $\widetilde X$ (see
the proof of Theorem 1.10 in \cite[p. 38]{Mu}). Clearly,
$$
{\widetilde X} \,\subset\, {\widetilde X}'
$$
is an open subscheme, where ${\widetilde X}'$ is the quotient in
\eqref{e5}.

\begin{proposition}\label{prop2}
The natural morphism 
$$
f\, :\, {\widetilde X}\, \longrightarrow\, X
$$
is an \'etale Galois covering with Galois group $\Gamma\,=\,
{\rm Gal}(\varphi)$. Moreover the
restriction
$$
\varphi\vert_{\widetilde M}\,:\, {\widetilde M} \,\longrightarrow\, M^{ss}
$$
is the pullback of $f$ via the quotient map $q\,:\,
M^{ss}\,\longrightarrow\, M/\!\!/G\,=\, X$.
\end{proposition}

\begin{proof}
{}From \eqref{e2} it follows immediately that the restriction of the covering
$\varphi$ to any orbit $\theta(x\, ,G)\, \subset\, M$ is trivial.
In other words, the inverse image $\varphi^{-1}(\theta(x\, ,G))$ is a disjoint
union of copies of $\theta(x\, ,G)$. In view of Lemma \ref{lem1}, this implies
that the Galois group $\Gamma$ acts simply transitively on the set of 
connected components of $\varphi^{-1}(\theta(x\, ,G))$. In particular, for any
$y\,\in\, M'$, the restriction of $\varphi$ to the orbit $\theta'(y\, ,G)$
is injective.

Therefore, to prove the proposition, we may replace $M$ by
the spectrum of an integral finite type algebra $A$, with quotient field
$K$, equipped with an action of $G$. Similarly, the variety ${\widetilde M}$ and
the action of $G$ on it are replaced by a connected finite \'etale algebra
$B$, with the quotient field $L$, over $A$ with Galois group $\Gamma$,
and a lifting to $B$ of the action of $G$ on $A$ that commutes with the
action of $\Gamma$ on $B$. The quotients $\widetilde X$ and $X$ get replaced by
$\text{Spec}(A^G)$ and $\text{Spec}(B^G)$ respectively. Since $M$ is smooth,
hence normal, and the map $\varphi$ restricted to any closed orbit of $G$ is
injective, the following lemma completes the proof of the proposition.
\end{proof}

\begin{lemma}\label{affine}
Suppose the $G$--equivariant finite \'etale
map $f\,:\, {\widetilde M} \,\longrightarrow \,M$ of affine varieties is
defined by an inclusion $A\,\subset\, B$ of finite type $k$ algebras
such that
\begin{itemize}
\item $A$ is normal, and

\item $f$ restricted to each closed
orbit of $G$ is an injection.
\end{itemize}
Then the induced map on the quotients ${\widetilde X}\,\longrightarrow\, X$ is also
finite \'etale, and ${\widetilde M}$ is the fiber product ${\widetilde X}\times_X M$.
\end{lemma}

\begin{proof}
This can be found in \cite{Dr} (see \cite[Proposition 4.16]{Dr} and
\cite[Proposition 4.18]{Dr}). In \cite{Dr} it is assumed that the characteristic of
the field $k$ is zero. However, the proof can be checked to be characteristic free. 
For the convenience of the reader, we give a brief outline of the proof.

The actions of $G$ on $A$ and $B$ 
extend to the quotient fields $K$ and $L$ respectively. The spaces of invariants
for the action of the Galois group $\Gamma$ on $B$ and $L$ are $A$ and $K$
respectively.

Since the actions of $G$ and $\Gamma$ on $B$ commute, the natural inclusion $A^G
\,\subset\, (B^G)^{\Gamma}$ is an isomorphism. This implies that $B^G$ is
finite over $A^G$. 

We first observe that $B^G$ is the integral closure of $A^G$ in
$L$. Indeed, if an element $a\,\in\, L$ is integral over $A^G$, then
all $G$ translates of $a$ are also solutions of the same equation. Therefore,
the connectedness of $G$ implies that $a$ is fixed by $G$.

Consequently, ${\rm Spec}(B^G)$ is a normal affine variety equipped with an action
of $\Gamma$ such that the quotient is ${\rm Spec}(A^G)$. So
to show that the map ${\widetilde X}\,\longrightarrow\, X$ in the lemma is \'etale it
is enough to check that this action of $\Gamma$ is free on the closed points. 

Let $x\,\in\, {\rm Spec}(B^G)$ be a closed point such that there is an element
$\gamma \,\in \,\Gamma$ with $\gamma\cdot x \,= \,x$. Let $y$ be a
closed point in the unique closed orbit in the fiber of the map
${\widetilde M} \,\longrightarrow\,{\widetilde X}$ over $x$. Since $\gamma$
commutes with $G$ (see Lemma \ref{lem1}) we get that $\theta'(\gamma\cdot y
\, ,G)\,=\, \gamma\cdot \theta'(y\, ,G)$ is also the unique closed orbit projecting
to $x$. Hence, we have $\theta'(\gamma\cdot y
\, ,G)\,=\, \theta'(y\, ,G)$. Now from the injectivity of the map from $\Gamma$
to the permutations of the components of $\theta'(y\, ,G)$ we conclude 
that $\gamma\cdot y \,=\,y$. Consequently, we have $\gamma \,=\,e$.
This proves that the morphism ${\widetilde X} \,\longrightarrow\, X$ is \'etale.

For the isomorphism ${\widetilde M}\,=\,{\widetilde X}\times_X M$, first note
that the $G$--equivariant inclusion $A \,\hookrightarrow\, B$ factors
via
$$
A \,\subset\, A \otimes _{A^G}B^G\, .
$$
Therefore, in order to prove that ${\widetilde M}\,=\,{\widetilde X}\times_X M$ it is 
enough to prove it under the assumption that the natural $G$--equivariant homomorphism
$$
A \otimes_{A^G}B^G \,\longrightarrow \,B
$$
is an isomorphism.

By using the conclusion of the first part of the lemma that $B^G$ is finite
and \'etale over $A^G$ of cardinality $|\Gamma|$, the base change $A
\,\hookrightarrow\, A\otimes_{A^G}B^G$ is also finite and \'etale of
the same cardinality. Since we started with a finite and \'etale algebra
$B$ over $A$ we conclude that $ A\otimes_{A^G}B^G \,\hookrightarrow \,B $
is also finite and \'etale. 

Now the isomorphism ${\widetilde M}\,=\,
{\widetilde X}\times_X M$ follows because both $\text{Spec}\,B$ and
$\text{Spec}(A\otimes_{A^G}B^G)$ are finite and \'etale over ${\rm Spec}\, A$
of same fiber cardinality $|\Gamma|$.
\end{proof} 

The construction in Proposition \ref{prop2} of a covering ${\widetilde X}$
starting from an \'etale covering $M'$ of $M$ is functorial (compatible with
the standard operations on coverings), and defines a homomorphism
\begin{equation}\label{e6}
H\, :\, \pi_1(X,\, x_0)\, \longrightarrow\, \pi_1(M,\, u_0)\, .
\end{equation}

Given an \'etale Galois covering $\phi\, :\, Y\, \longrightarrow\, X$, the
\'etale Galois covering of $M$ corresponding to the homomorphism $h$ in \eqref{f1}
is constructed as follows. Consider the pullback of the covering $\phi$ to the
open subset of $M$ where the rational map $M\,\dasharrow\, X$ is defined. This
covering extends to $M$ because the complement of the open subset has codimension
at least two and $M$ is smooth.

The above two constructions, namely the construction of a covering of $X$ from
a covering of $M$, and the construction of a covering of $M$ from
a covering of $X$, are clearly inverses of each other. Therefore,
for the two homomorphisms $h$ and $H$ in \eqref{f1} and \eqref{e6},
we have $h\circ H\, =\, \text{Id}_{\pi_1(X, x_0)}$ and
$H\circ h\, =\, \text{Id}_{\pi_1(M, u_0)}$. Consequently, the following is proved:

\begin{theorem}\label{thm1}
The homomorphism $h$ in \eqref{f1} is an isomorphism.
\end{theorem}

\begin{remark}
Let $M\,=\, {\mathbb P}^1_k/\{0\,=\, \infty\}$ be the unique nodal curve of
arithmetic genus one. The action of ${\mathbb G}_m$ on ${\mathbb P}^1_k$ defined by
$t\cdot (x_1\, ,x_2)\,=\, (tx_1\, ,x_2/t)$ produces an action of ${\mathbb G}_m$
on $M$. While $M$ is not simply connected, the GIT quotient $M/\!\!/{\mathbb G}_m
\,=\,{\rm Spec}\,k$ is so. So the condition in Theorem \ref{thm1} that $M$
is smooth is essential.
\end{remark}

\section{The topological fundamental group}

In this section we assume that $k\,=\, \mathbb C$. The topological
fundamental group of any $Y$ with base point $y_0\,\in\, Y$ will be denoted
by $\pi^t_1(Y, \, y_0)$. We recall that an irreducible smooth complex projective
variety is also called a complex projective manifold.

As before, fix a $G$--equivariant algebraic line bundle on the complex
projective manifold $M$.

Let
$$
\varphi\, :\, M'\, \longrightarrow\, M
$$
be a holomorphic \'etale Galois covering of $M$ such that $M'$ is connected. It
should be clarified that the degree of $\varphi$ is now allowed to be infinite.

\begin{proposition}\label{prop3}
\mbox{}
\begin{enumerate}
\item There is a unique holomorphic action of $G$ on $M'$
$$
\theta'\, :\, M'\times G\, \longrightarrow\, M'
$$
that lifts the action $\theta$ of $G$ on $M$.

\item The actions of $G$ and $\Gamma\,=\, {\rm Gal}(\varphi)$ on $M'$ commute.
\end{enumerate}
\end{proposition}

\begin{proof}
The proof of the first (respectively, second) part of the proposition is exactly
identical to the proof of Proposition \ref{prop1} (respectively, Lemma \ref{lem1}).
\end{proof}

The geometric invariant theory is not applicable for the action of $G$ on $M'$
because $M'$ is not an algebraic variety in general. However we will
construct from $M'$ a covering of the GIT quotient $X\,=\, M/\!\!/ G$.

Let $M^{ss}\, \subset\, M$ be the semistable locus for the
action of $G$ on $M$; it is a Zariski open subset
preserved by the action of $G$. Let
\begin{equation}\label{beta}
\beta\, :\, M^{ss}\,\longrightarrow\, X
\end{equation}
be the quotient map. This map $\beta$ is surjective.

Take an affine open subset ${\mathcal U}\, \subset\, X$. The inverse image
$$
\beta^{-1}({\mathcal U})\, \subset\, M^{ss}
$$
is also an affine open subset. Fix a maximal compact subgroup
$$
K_G\, \subset\, G\, .
$$
There is a $K_G$--invariant subset
$$
U_K\, \subset\, \beta^{-1}({\mathcal U})
$$
such that
\begin{itemize}
\item $\beta^{-1}({\mathcal U})$ admits a deformation retraction to $U_K$,

\item the map $\beta\vert_{U_K}\, :\, U_K\, \longrightarrow\, \mathcal U$ is
surjective,

\item the quotient map $U_K/K_G\, \longrightarrow\, \mathcal U$ is a homeomorphism, and

\item the subset $U_K$ is contained in the union of all closed $G$--orbits 
satisfying the condition that the intersection with $U_K$ is a $K$--orbit.
\end{itemize}
(See \cite[p. 422, Corollary 1.4]{Ne} and \cite[p. 424, Theorem 2.1]{Ne}; also stated
in the first paragraph of the introduction in \cite[p. 419]{Ne}. See also \cite{KN}.)

Let
\begin{equation}\label{s2}
\widetilde{U}_K\,:=\, \varphi^{-1}(U_K)\, \subset\, M'
\end{equation}
be the inverse image. As $\beta^{-1}({\mathcal U})$ is a nonempty Zariski open subset
of $M$, and $M$ is smooth, the homomorphism of topological fundamental groups
\begin{equation}\label{s1}
\pi^t_1(\beta^{-1}({\mathcal U}))\, \longrightarrow\, \pi^t_1(M)\, ,
\end{equation}
induced by the inclusion of $\beta^{-1}({\mathcal U})$ in $M$, is surjective.
Since $U_K$ is a deformation
retraction of $\beta^{-1}({\mathcal U})$, from the surjectivity of the homomorphism
in \eqref{s1} it follows immediately that the homomorphism
\begin{equation}\label{s-1}
\pi^t_1(U_K)\, \longrightarrow\, \pi^t_1(M)
\end{equation}
induced by the inclusion of $U_K$ in $M$ is surjective. Note that surjectivity of the
homomorphism induced on fundamental groups induced by a map is equivalent to the 
condition that the pullback 
to the domain (of the map) of any connected \'etale covering of the target space 
(of the map) remains connected. Consequently, from the surjectivity of
the homomorphism in \eqref{s-1} it follows that the inverse
image $\widetilde{U}_K$ in \eqref{s2} is connected (recall that $M'$ is connected).

\begin{lemma}\label{lem2}
Take a point $x\, \in\, M$. Let $Z\, :=\, \theta(x\, ,K_G)\, \subset\, M$ be the
$K_G$--orbit of $x$. Then $\varphi^{-1}(Z)$ is a disjoint union of copies of $Z$.
More precisely, the restriction of $\varphi$ to each connected component $S$ of
$\varphi^{-1}(Z)$ is a homeomorphism from $S$ to $Z$.
\end{lemma}

\begin{proof}
Let $\iota^x\, :\, Z\, \hookrightarrow\, M$ be the inclusion map. To prove the
lemma it suffices to show that the homomorphism
$$
\iota^x_*\, :\, \pi^t_1(Z,\, x)\, \longrightarrow\, \pi^t_1(M,\, x)\, ,
$$
induced by the inclusion $\iota^x$, is trivial.

Let $G_x\, \subset\, G$ be the isotropy subgroup for $x$. The $G$--orbit $\theta
(x\, ,G)$ is identified with the quotient $G/G_x$. Since $M$ is projective, there is
an irreducible smooth complex projective variety $\widetilde{Z}$ containing 
$G/G_x$ as a Zariski open subset such that the inclusion map
$$
G/G_x\,=\, \theta(x\, ,G)\,\hookrightarrow\, M
$$
extends to a morphism
\begin{equation}\label{tau}
\tau\, :\,\widetilde{Z}\, \longrightarrow\, M\, .
\end{equation}

Since $G$ is a rational variety, the quotient
$G/G_x$ is unirational. So $\widetilde{Z}$ is unirational. Since $\widetilde{Z}$
is smooth, this implies that $\widetilde{Z}$ is simply
connected \cite[p. 483, Proposition 1]{Se}.

The above inclusion map $\iota^x$ coincides with the composition
$$
Z\,=\, K_G/(K_G\cap G_x) \, \hookrightarrow\, G/G_x \, \hookrightarrow\, \widetilde{Z}
\, \stackrel{\tau}{\longrightarrow}\, M\, ,
$$
where $\tau$ is the map in \eqref{tau}. Since $\widetilde{Z}$ is simply 
connected, this implies that $\iota^x_*\,=\, 0$.
As noted before, the lemma follows from this.
\end{proof}

\begin{corollary}\label{cor1}
The image in $\pi^t_1(M)$ of the fundamental group of any $G$ orbit in $M$ is trivial.
\end{corollary}

\begin{proof}
We saw in the proof of Lemma \ref{lem2} that the inclusion of the orbit
$\theta(x\, ,G)$ in $M$ factors through the simply connected variety $\widetilde{Z}$.
\end{proof}

Define
\begin{equation}\label{e7}
\widetilde{U}\, :=\, \widetilde{U}_K/K_G\, ,
\end{equation}
where $\widetilde{U}_K$ is defined in \eqref{s2}.
We note that $\widetilde{U}$ is connected because $\widetilde{U}_K$ is connected.
The natural map 
$$
(\beta\circ \varphi)\vert_{\widetilde{U}_K}\, :\, \widetilde{U}_K\,
\longrightarrow\, \mathcal U\, ,
$$
is clearly $K_G$--invariant. So it descends to a map
\begin{equation}\label{a3}
\widetilde{\varphi}_U\, :\, \widetilde{U}\, \longrightarrow\, \mathcal U\, ,
\end{equation}
where $\widetilde{U}$ is defined in \eqref{e7}.

Since the actions of $G$ and $\Gamma$ on $M'$ commute (Proposition
\ref{prop3}(2)), the action of $\Gamma$ on $\widetilde{U}_K$ descends to
an action of $\Gamma$ on the quotient $\widetilde{U}$ in \eqref{e7}. Note that the
quotient map
\begin{equation}\label{a2}
\widetilde{U}\, \longrightarrow\, \widetilde{U}/\Gamma
\end{equation}
is an \'etale Galois covering map with Galois group $\Gamma$ because the covering
$$
\varphi\vert_{\widetilde{U}_K}\, :\, \widetilde{U}_K\,\longrightarrow\, U_K
$$
is $K_G$--equivariant and the map in \eqref{a2} is the quotient for the
actions of $K_G$.

Since
$$
({\widetilde U}_K/\Gamma)/K_G\,=\, \mathcal{U} \ ~ \text{ and }~ \
({\widetilde U}_K/\Gamma)/K_G\,=\, ({\widetilde U}_K/K_G)/\Gamma\, ,
$$
we have $\widetilde{U}/\Gamma\,=\, {\mathcal U}$. Therefore, the map
$\widetilde{\varphi}_U$ in \eqref{a3} is an \'etale Galois covering with
Galois group $\Gamma$.

The intersection of two affine open subsets of $X$ is again an affine open subset.
Take another affine open subset ${\mathcal V}\, \subset\, X$, and define
${\mathcal W}\,=\, {\mathcal U}\bigcap {\mathcal V}$. Construct $\widetilde{V}$
(respectively, $\widetilde{W}$), as done in \eqref{e7}, from ${\mathcal V}$
(respectively, ${\mathcal W}$).

Identify a point $z$ of $\widetilde{U}$ with a point $z'$ of $\widetilde{V}$ if the
$K_G$--orbit in $\widetilde{U}_K$ for $z$ and the $K_G$--orbit in $\widetilde{V}_K$
for $z'$ lie in the same $G$--orbit in $M'$. Note that the subset of $\widetilde{U}$
consisting of points whose $K_G$--orbit lie in $\theta'(\widetilde{V}\, ,G)$ actually
coincides with the subset $\widetilde{W}\, \subset\, \widetilde{U}$.

Take a finite collection of affine open subsets $\{{\mathcal U}_i\}_{i=1}^n$
of $X$ that cover $X$. Let $\widetilde{U}_i$ be the topological space constructed
as in \eqref{e7} from ${\mathcal U}_i$. Let $X'$ be the topological space constructed
by performing the above gluing on the disjoint union
$\bigsqcup_{i=1}^n \widetilde{U}_i$. Recall
that all $\widetilde{U}_i$ are connected, and for any pair $\widetilde{U}_i$ and
$\widetilde{U}_j$, the gluing between them is done over
connected open subsets of $\widetilde{U}_i$ and $\widetilde{U}_j$. Therefore, the
resulting topological space $X'$ is connected.

Recall that each $\widetilde{U}_i$ is an \'etale Galois covering of ${\mathcal U}_i$
with Galois group $\text{Gal}(\varphi)$ (via the map
$\widetilde{\varphi}_U$ in \eqref{a3}). The intersection of $\widetilde{U}_i$ with
$\widetilde{U}_j$ inside $X'$ is an \'etale Galois covering of ${\mathcal U}_i\bigcap
{\mathcal U}_j$ with Galois group $\text{Gal}(\varphi)$. Therefore, $X'$
is an \'etale Galois covering of $X$ with Galois group $\text{Gal}(\varphi)$.

Just as Theorem \ref{thm1} was proved, we now have the following:

\begin{theorem}\label{thm2}
The homomorphism between topological fundamental groups induced by the
rational map $M\, \dasharrow\, M/\!\!/G$ is an isomorphism.
\end{theorem}

\begin{remark}
It is easy to construct examples to show that Theorem \ref{thm0} is false if $M$ 
is not assumed to be projective. Let $G\,=\,{\mathbb G}_m$ act on $M\,=\,{\mathbb 
G}_m$ by left translations. Then the quotient $M/\!\!/G$ coincides with ${\rm 
Spec}\,k$. Clearly $\pi_1({\mathbb G}_m)\,\longrightarrow\, \pi_1({\rm Spec}\, k)$ 
is not an isomorphism. 
\end{remark}

\section{Acknowledgements}

Niels Borne pointed out a gap in the proof of Proposition \ref{prop1} given in
an earlier version. We are very grateful to him for this.
We are very grateful to Jakob Stix for some very helpful correspondences. We
thank Michel Brion for a useful correspondence. We are very grateful to the
two referees for their comments. We thank the managing editors for pointing
out \cite{Li}. The first-named author
acknowledges the support of the J. C. Bose Fellowship.


\end{document}